\documentclass[12pt]{amsart}

\usepackage{amsmath}
\usepackage{amssymb}
\usepackage{amsfonts}
\usepackage{amsthm}
\usepackage{mathtools}
\usepackage{relsize}
\newtheorem*{theorem*}{Theorem}
\usepackage{tikz}
\usepackage{tikz-cd}
\usepackage{enumitem}
\usepackage[margin=1in]{geometry}
\usepackage{xcolor}
\usepackage[colorlinks=true,linkcolor=blue,citecolor=blue]{hyperref}

\numberwithin{equation}{section}
\newtheorem{theorem}{Theorem}[section]
\newtheorem{lemma}[theorem]{Lemma}
\newtheorem{proposition}[theorem]{Proposition}
\newtheorem{corollary}[theorem]{Corollary}
\theoremstyle{definition}

\theoremstyle{remark}
\newtheorem{remark}[theorem]{Remark}

\title[Smooth structure sets of products of spheres]{Smooth structure sets of products of spheres\\ and manifolds homotopy equivalent to them}
\author{Sagnik Biswas}
\address{Department of Mathematics, Indian Institute of Technology Madras, Chennai 600036, India} 
\email{ma20d013@smail.iitm.ac.in}

\date{}

\begin{document}
\begin{abstract}

We classify, up to almost diffeomorphism, the smooth closed oriented manifolds
homotopy equivalent to each of the sphere products $S^{4k-1}\times S^{4k}$,
$S^{4k}\times S^{4k}$, $S^{4k}\times S^{4k+1}$.  In each case we realize the image of the normal-invariant map in the smooth surgery exact sequence by explicit elements in the structure sets, and use that to specifically compute the structure sets modulo the action of homotopy spheres. We also give the almost diffeomorphism classification of manifolds in the homotopy type of the above manifolds.
\end{abstract}

\maketitle

\section{Introduction}\label{sec:intro}

The study of smooth manifolds homotopy equivalent to a product of spheres
started with Novikov~\cite{novikov}. For $S^p\times S^q$ 
the smooth structure set  $hS(S^p\times S^q)$, is investigated by Crowley~\cite{crowley-sss}, where he fits it into an exact sequence; the full structure set is evaluated in the lowest-dimensional case explicitly but in general no explicit geometric constructions were given. The underlying diffeomorphism types are described in general for
$n-1$ connected $2n$ or $2n+1$ manifolds by Wall, Wilkens and
Jiang-Su~\cite{wall-vi,wilkens,jiang-su}, the latter provided a geometric description for $n=4k-1$, but without isolating to a given sphere product's homotopy type. There are some results about products of spheres that work with a stronger equivalence. De Sapio,
Kawakubo and Schultz~\cite{desapio-prod,kawakubo,schultz} classify the manifolds
\emph{homeomorphic} to $S^p\times S^q$.

We investigate in this paper smooth manifolds homotopy equivalent to product of spheres $S^{n} \times S^n$ and $S^n \times S^{n+1}$. If none of the $n,n+1$ is of the form $4k$, it is equivalent to investigate smooth manifolds  tangentially homotopy equivalent to $S^{n} \times S^n$ and $S^n \times S^{n+1}$. And, in that case it is known to have only one smooth manifold, product of spheres itself. So, the investigation boils down to the case where at least one of $n$ or $n+1$ is of the form $4k$. This paper gives a complete and explicit classification, up to almost
diffeomorphism, of the smooth closed oriented manifolds homotopy equivalent to
$S^{4k-1}\times S^{4k}$, $S^{4k}\times S^{4k}$, $S^{4k}\times S^{4k+1}$. Then, we also find their smooth structure set modulo action by homotopy spheres, which is also of independent interest. In each of the above cases we determine the image of the normal-invariant map when needed, realise every normal invariant by an
explicit structure set element, and thus evaluating the
structure set and classifying the manifolds up to almost diffeomorphism. Concretely, for $k\ge1$ (and $k\ge2$ in the $S^{4k-1} \times S^{4k}$-case) every such manifold
is almost diffeomorphic to: a sphere bundle $X_{l(c)}=S(\xi_{l(c)})$ over $S^{4k}$
for a unique $c\ge0$ (Theorem~\ref{thm:main-result}); a plumbing-type manifold
$X_{(u,v)}$ indexed by the realisable pairs $\partial(u,v)=0$
(Theorem~\ref{thm:classification}); $S(m\eta)$ for some
$m\in\mathbb{Z}$ (Theorem~\ref{thm:classification-odd}).

\section{Smooth manifolds homotopy equivalent to $S^{4k-1}\times S^{4k}$}

Let's define,
\begin{equation}
a_k \;=\; \begin{cases} 1 & k\text{ even}\\ 2 & k\text{ odd}\end{cases},\qquad
b_{4k-1} \;=\; \begin{cases} 2 & 4k-1 \in \{3,7\}\\ 1 & \text{otherwise}\end{cases},\qquad
y_k \;:=\; a_k\cdot b_{4k-1}\cdot (2k-1)!.
\end{equation}

By Jiang-Su~\cite{jiang-su}, for every $l \in y_k\,\mathbb{Z}$ there is a unique $(4k)$-dimensional real vector bundle $\xi_l$ over $S^{4k}$ with $e(\xi_l) = 0$ and $p_k(\xi_l) = l\cdot\omega$, where $\omega$ generates $H^{4k}(S^{4k};\mathbb{Z})$. Let $X_l := S(\xi_l)$ denote the total space of the associated $S^{4k-1}$-bundle. Note that $X_0 = S^{4k-1}\times S^{4k}$. The main theorems of this section is Theorem \ref{thm:structure-prod}.

\begin{theorem}\label{thm:main-result}
Let $k \geq 1$, and set
\begin{equation}
l(c) \;:=\; c\cdot a_k\cdot (2k-1)!\cdot \operatorname{den}\!\bigl(B_k/(4k)\bigr), \qquad
t_{4k} \;:=\; a_k\cdot 2^{2k-2}(2^{2k-1}-1)\cdot \operatorname{Num}(B_k/(4k)).
\end{equation}
For each $c \in \mathbb{Z}$, the fibre-homotopy equivalence $f_{l(c)}\colon X_{l(c)} \xrightarrow{\simeq} S^{4k-1}\times S^{4k}$ (provided by Lemma~\ref{thm:main}) has normal invariant whose image in $[S^{4k}, G/\textnormal{Top}] = \mathbb{Z}$ is $\pm c\cdot t_{4k}$; the family $\{f_{l(c)}\}_{c \in \mathbb{Z}}$ realises every element of $t_{4k}\mathbb{Z}$.

\end{theorem}
\smallskip
\begin{proof}
We prove this from the fibre-homotopy equivalence criterion 
(Lemma~\ref{thm:main}).

\begin{lemma}[Fibre-homotopy equivalence criterion]\label{thm:main}
Let $k \geq 1$ and $l \in y_k\,\mathbb{Z}$. The sphere bundle $\xi_l$ is fibre-homotopy trivial if and only if
\begin{equation}\label{eq:divisibility}
l \;\in\; a_k\cdot (2k-1)!\cdot \operatorname{den}\!\bigl(B_k/(4k)\bigr)\cdot \mathbb{Z}.
\end{equation}
\end{lemma}

\noindent When \eqref{eq:divisibility} holds, we write $f_l \colon X_l \xrightarrow{\simeq} S^{4k-1}\times S^{4k}$ for the resulting fibre-homotopy equivalence.

\begin{proof}[Proof of necessity]
Assume $X_l \simeq_{\textnormal{fhe}} S^{4k-1}\times S^{4k}$. Then $\xi_l$ has trivial associated spherical fibration, so the clutching map $\xi_l \in \pi_{4k-1}(SO(4k))$ lies in the kernel of the map $A_*\colon \pi_{4k-1}(SO(4k)) \to \pi_{4k-1}(SG(4k))$.

Let $\alpha_l \in \pi_{4k-1}(SO) = \mathbb{Z}$ denote the stable class of $\xi_l$. Since stabilisation preserves Pontryagin classes, $p_k(\alpha_l) = l$. The map $p_k\colon \pi_{4k-1}(SO) = \mathbb{Z} \to \mathbb{Z}$ is injective with image $a_k\cdot (2k-1)!\cdot \mathbb{Z}$ \cite[Lem.~1.1(ii)]{kervaire-note}. By naturality of $J$ under stabilisation, $A_*(\xi_l) = 0$ implies $J_s(\alpha_l) = 0$ in $\pi_{4k-1}^s$. The kernel of $J_s\colon \pi_{4k-1}(SO) \to \pi_{4k-1}^s$ is the subgroup $\operatorname{den}(B_k/(4k))\cdot \mathbb{Z}$, by \cite{adams-jiv}.

Since $\alpha_l \in \ker J_s$, we have $\alpha_l = n\cdot \operatorname{den}(B_k/(4k))\cdot g$ for some $n \in \mathbb{Z}$, where $g$ generates $\pi_{4k-1}(SO)$. Applying $p_k$:
\begin{equation}
l \;=\; p_k(\alpha_l) \;=\; n \cdot \operatorname{den}(B_k/(4k)) \cdot a_k(2k-1)!.
\end{equation}
Hence $l \in a_k(2k-1)!\operatorname{den}(B_k/(4k))\mathbb{Z}$. This proves the necessity part. \phantom\qedhere
\end{proof}

\begin{proof}[Proof of sufficiency]
Assume $l \in a_k(2k-1)!\operatorname{den}(B_k/(4k))\mathbb{Z}$. We show $A_*(\xi_l) = 0$ in \\ $\pi_{4k-1}(SG(4k))$. Since $e(\xi_l) = 0$, we have $\xi_l = \xi'_l \oplus \epsilon$ uniquely for some $\xi'_l \in \pi_{4k-1}(SO(4k-1))$.  Writing $J^u\colon \pi_{4k-1}(SO(4k-1)) \to \pi_{4k-1}(SF(4k-1)) \cong \pi_{8k-2}(S^{4k-1})$ for the unstable $J$-homomorphism and $\psi\colon SF(4k-1) \to SG(4k)$ for the natural inclusion, commutativity gives $A_*(\xi_l) = \psi_*(J^u(\xi'_l))$.

\noindent Set $\beta := J^u(\xi'_l) \in \pi_{8k-2}(S^{4k-1})$. By commutativity of the stabilization, $\beta$ stabilises to $J_s(\alpha_l) = 0$ in $\pi_{4k-1}^s$. Stabilisation factors as
\begin{equation}
\pi_{8k-2}(S^{4k-1}) \xrightarrow{E_1} \pi_{8k-1}(S^{4k}) \xrightarrow{E_2} \pi_{8k}(S^{4k+1}) \cong \pi_{4k-1}^s.
\end{equation}
The EHP sequence \cite{hilton-whitehead} gives $\ker E_2 = \mathbb{Z}\cdot[\iota_{4k}, \iota_{4k}]$ and $\ker E_1 = P_1(\pi_1^s) \le \mathbb{Z}/2$. Since $E_2 E_1(\beta) = 0$, we have $E_1(\beta) \in \ker E_2 = \mathbb{Z}\cdot[\iota,\iota]$. The Hopf invariant satisfies $H(E_1\beta) = e(\xi_l) = 0$, while $H[\iota,\iota] = \pm 2 \neq 0$, so
\begin{equation}
E_1(\beta) \;\in\; \mathbb{Z}\cdot[\iota,\iota]\,\cap\,\ker H \;=\; 0,
\end{equation}
and hence $\beta \in \ker E_1$.

The fibration $SF(4k-1) \to SG(4k) \xrightarrow{\textnormal{ev}} S^{4k-1}$ gives $\ker \psi_* = \operatorname{Im}(\partial)$ where $\partial(\alpha) = [\iota_{4k-1}, \alpha]$, so $\operatorname{Im}(\partial) = \mathbb{Z}/2\cdot[\iota_{4k-1}, \eta_{4k-1}]$. By the James--Whitehead formula \cite{hilton-whitehead}, $\ker E_1 = P_1(\pi_1^s) = \mathbb{Z}/2\cdot[\iota_{4k-1}, \eta_{4k-1}]$, so $\ker E_1 \subseteq \ker \psi_*$. Therefore $\psi_*(\beta) = 0$, i.e.\ $A_*(\xi_l) = 0$. This proves the sufficiency. \phantom\qedhere
\end{proof}

\begin{proof}[\textbf{Proof of Theorem~\ref{thm:main-result}}]
Since $X_{l(c)} = S(\xi_{l(c)})$ and $f_{l(c)}\colon X_{l(c)} \xrightarrow{\simeq} S^{4k-1}\times S^{4k}$ is a fibre-homotopy equivalence over $S^{4k}$, it corresponds to a pair $(\xi_{l(c)}, f_{l(c)}) \in [S^{4k}, G/O]$, represented by the degree-1 normal map $\gamma_c \in N^O(S^{4k})$:
\begin{equation}
\begin{array}{ccc}
\nu_W & \longrightarrow & \xi_{l(c)} \\
\big\downarrow & & \big\downarrow \\
W & \xrightarrow{\;f\;} & S^{4k},
\end{array}
\end{equation}
 By Crowley-Escher's Lemma~5.2 \cite[Lem.~5.2]{crowley-escher}, the normal invariant of $f_{l(c)}$ then satisfies
\begin{equation}
\eta_O(f_{l(c)}) \;=\; \pi_X^*\,\gamma_c \quad \text{in } [X, G/O],
\end{equation}
where $\pi_X\colon X = S^{4k-1}\times S^{4k} \to S^{4k}$ is the second-factor projection.

The image of $\eta_O(f_{l(c)})$ in $[S^{4k}, G/\textnormal{Top}] = \mathbb{Z}$ is (up to sign) the surgery obstruction $\sigma(\gamma_c)$. By Hirzebruch's signature theorem~\cite{hirzebruch},
\begin{equation}
8\,\sigma(\gamma_c) \;=\; \operatorname{sgn}(W) \;=\; L_k(p_1^W, \ldots, p_k^W)[W].
\end{equation}
By naturality of Pontryagin classes, $p_i^W = f^*p_i(\xi_{l(c)}) = 0$ for $i < k$ and $p_k^W = f^*p_k(\xi_{l(c)}) = l(c)\cdot \omega \in H^{4k}(S^{4k})$ (using $\deg f = 1$). Only the coefficient of $p_k$ in $L_k$ contributes:
\begin{equation}
8\,\sigma(\gamma_c) \;=\; \operatorname{(coeff. \ of \ }p_k \text{ in } L_k)\cdot p_k^W[W] \;=\; \frac{2^{2k}(2^{2k-1}-1)B_k}{(2k)!}\cdot l(c).
\end{equation}
Substituting $l(c) = c\cdot a_k(2k-1)!\operatorname{den}(B_k/(4k))$, using $(2k)! = 2k\cdot(2k-1)!$, and applying \\ $B_k\cdot \operatorname{den}(B_k/(4k)) = 4k\cdot \operatorname{Num}(B_k/(4k))$:
\begin{equation}
\sigma(\gamma_c) \;=\; c\cdot a_k\cdot 2^{2k-2}(2^{2k-1}-1)\cdot \operatorname{Num}(B_k/(4k)) \;=\; c\cdot t_{4k}.
\end{equation}

The image of $\rho\colon [S^{4k}, G/O] \to [S^{4k}, G/\textnormal{Top}] = \mathbb{Z}$ is exactly $t_{4k}\mathbb{Z}$ \cite[Prop.~13.23]{ranicki-ags}, and the family $\{[X_{(l(c))},f_{(l(c))} ] : c \in \mathbb{Z}\} \subseteq hS(S^{4k-1} \times S^{4k}) $ exhausts this out bijectively.  \qedhere
\end{proof}  \phantom\qedhere
\end{proof} 

\begin{proposition}\label{prop:pinch-prod}
\begin{enumerate}
  \item[\textnormal{(i)}] Let $S^{8k-1} \xrightarrow{\;\alpha\;} S^{4k-1} \xrightarrow{\;h_1\;} S^{4k-1}\times S^{4k}$ be a map and let $P_{(h_1\circ\alpha)}$ be the corresponding pinch map.  Then
  \begin{equation}\label{eq:pinch-alpha-prod}
    i^*\eta\bigl(P_{(h_1\circ\alpha)}\bigr) \;=\; \bigl(0,\; \pm\tau_*\{\alpha\}\bigr).
  \end{equation}
  \item[\textnormal{(ii)}] Let $S^{8k-1} \xrightarrow{\;\beta\;} S^{4k} \xrightarrow{\;h_2\;} S^{4k-1}\times S^{4k}$ be a map and let $P_{(h_2\circ\beta)}$ be the corresponding pinch map.  Then
  \begin{equation}\label{eq:pinch-beta-prod}
    i^*\eta\bigl(P_{(h_2\circ\beta)}\bigr) \;=\; \bigl(\pm\tau_*\{\beta\},\; \tau_*\{\beta\circ\eta\}\bigr) \quad\text{or}\quad \bigl(\pm\tau_*\{\beta\},\; 0\bigr).
  \end{equation}
  \item[\textnormal{(iii)}] Every element of $Im\bigl([S^{4k-1}, SG] \oplus [S^{4k}, SG]\bigr)$ is realised as $i^*\eta\bigl(P_{(h_2\circ\beta)} \circ P_{(h_1\circ\alpha)}\bigr)$ for suitable $\alpha \in \pi_{8k-1}(S^{4k-1})$ and $\beta \in \pi_{8k-1}(S^{4k})$.
  \item[\textnormal{(iv)}] Every element of $[S^{4k-1}, G/O] \oplus [S^{4k}, G/O]$ is realised as $i^*\eta\bigl(P_{(h_2\circ\beta)} \circ P_{(h_1\circ\alpha)} \circ f_{l(c)}\bigr)$ for suitable $\alpha \in \pi_{8k-1}(S^{4k-1})$ and $\beta \in \pi_{8k-1}(S^{4k}), c\in \mathbb{Z}$.
\end{enumerate}
\end{proposition}

\begin{proof}[proof of (i), (ii), (iii)]
    The proof follows exactly as in the proof in \cite[Thm.~3.2, 3.3]{biswas-vs} 
\end{proof}

\begin{proof}[proof of (iv)]
Let $X = S^{4k-1} \times S^{4k}$. For $\alpha \in \pi_{8k-1}(S^{4k-1})$, $\beta \in \pi_{8k-1}(S^{4k})$ and $c \in \mathbb{Z}$, put
\begin{equation}\label{eq:Adef-prod}
A_{\alpha,\beta,c}\colon\; X_{l(c)} \xrightarrow{\;f_{l(c)}\;} X \xrightarrow{\;P_{(h_1\circ\alpha)}\;} X \xrightarrow{\;P_{(h_2\circ\beta)}\;} X .
\end{equation}
Both pinch maps are the identity on the lower skeleton, so $(P^{-1})^*$ is the identity on $i^*[\,\cdot\,,G/O]$ and Madsen's composition formula~\cite[Cor.~2.6]{mtw} gives
\begin{equation}\label{eq:mad-comp-prod}
i^*\eta\bigl(A_{\alpha,\beta,c}\bigr) \;=\; i^*\eta\bigl(f_{l(c)}\bigr) \;+\; i^*\eta\bigl(P_{(h_1\circ\alpha)}\bigr) \;+\; i^*\eta\bigl(P_{(h_2\circ\beta)}\bigr).
\end{equation}
By Theorem~\ref{thm:main-result}(ii) we have $i^*\eta(f_{l(c)}) = (0,\gamma_c)$ with $F(\gamma_c) = c\,t_{4k}$.  Substituting \eqref{eq:pinch-alpha-prod} and \eqref{eq:pinch-beta-prod} into \eqref{eq:mad-comp-prod},
\begin{equation}\label{eq:sum-prod}
i^*\eta\bigl(A_{\alpha,\beta,c}\bigr) \;=\; \bigl(\pm\tau_*\{\beta\},\;\; \gamma_c \pm \tau_*\{\alpha\} + \tau_*\{\beta\circ\eta\}\bigr) \quad\text{or}\quad \bigl(\pm\tau_*\{\beta\},\;\; \gamma_c \pm \tau_*\{\alpha\}\bigr).
\end{equation}    

\noindent As, $\pi_{8k-1}(S^{4k-1}) \to \pi_{4k}^s$ and, $\pi_{8k-1}(S^{4k}) \to \pi_{4k-1}^s$ is onto \cite[Lemma 5.3(i)]{mtw}, this concludes the proof. 
\end{proof}

\begin{theorem}[The Structure Set]\label{thm:structure-prod}
The structure set is given as follows:
\begin{enumerate}
    \item The collection
    \begin{multline*}
    \Bigl\{\,\bigl[A_{\alpha,\beta,c}\colon X_{l(c)} \xrightarrow{\;f_{l(c)}\;} S^{4k-1}\times S^{4k} \xrightarrow{\;P_{(h_1\circ\alpha)}\;} S^{4k-1}\times S^{4k} \xrightarrow{\;P_{(h_2\circ\beta)}\;} S^{4k-1}\times S^{4k}\bigr] \;\Bigm|\; \\
    \alpha \in \pi_{8k-1}(S^{4k-1}),\; \beta \in \pi_{8k-1}(S^{4k}),\; c \in \mathbb{Z} \,\Bigr\}
    \end{multline*}
    exhausts $hS(S^{4k-1} \times S^{4k})/\Theta_{8k-1}$. \\ 

\item $[A_{\alpha,\beta,c}] = [A_{\alpha',\beta',c'}]$ in $hS(S^{4k-1}\times S^{4k})/\Theta_{8k-1}$ iff $c = c'$, $\tau_*\{\alpha\} = \tau_*\{\alpha'\}$ and $\tau_*\{\beta\} = \tau_*\{\beta'\}$.\\
\item Every smooth, closed, oriented manifold $M$ homotopy equivalent to $S^{4k-1}\times S^{4k}$ is almost diffeomorphic to $X_{l(c)}$ for a unique $c \in \mathbb{Z}_{\geq 0}$.
\end{enumerate}
\end{theorem}

\begin{proof}[proof of (1)]
Let, $X=S^{4k-1}\times S^{4k}$.
Take $[N,f]\in hS(X)$.
By Proposition~\ref{prop:pinch-prod}(iv) 
$i^{*}\eta([N,f])=i^{*}\eta\big([P_{(h_2\circ\beta)}\circ P_{(h_1\circ\alpha)}\circ f_{l(c)}]\big)$ for some $(\alpha,\beta,c)$.  This means,
$\eta([N,f])=\eta\big([A_{\alpha,\beta,c}]\big)+c^{*}(x)$, \quad where, $x\in\pi_{8k-1}(G/O)$,
and $x=\eta_S(\Sigma_x)$ for some homotopy sphere $\Sigma_x$, and, $\eta_S\colon\Theta_{8k-1}\to\pi_{8k-1}(G/O)$ is the normal invariant map for sphere, which is onto as $L_{8k-1}(\mathbb{Z})=0$.
So, $\eta([N,f])=\eta\big(\Sigma_x\cdot[A_{\alpha,\beta,c}]\big)$, which implies,
$[N,f]=(\Sigma_{bP}\cdot\Sigma_x)\cdot[A_{\alpha,\beta,c}]\in hS(X)$, for some $\Sigma_{bP}\in bP_{8k}$.
So, $[N,f]=[A_{\alpha,\beta,c}]\in hS(X)/\Theta_{8k-1}$.
This proves the realizability part. \phantom\qedhere
\end{proof}

\begin{proof}[proof of (2)]
Since the $\Theta_{8k-1}$-action changes the normal invariant only on the top cell, by $c^*(\eta_S(\Sigma))$,  we have
\begin{equation}\label{eq:bridge-prod}
[A_{\alpha,\beta,c}] = [A_{\alpha',\beta',c'}] \ \text{in} \ hS(X)/\Theta_{8k-1} \iff i^*\eta\bigl(A_{\alpha,\beta,c}\bigr) = i^*\eta\bigl(A_{\alpha',\beta',c'}\bigr).
\end{equation}
($\Rightarrow$)\; Let $[A_{\alpha,\beta,c}] = [A_{\alpha',\beta',c'}]$ in $hS(X)/\Theta_{8k-1}$.  Then by \eqref{eq:bridge-prod} and \eqref{eq:sum-prod},
\begin{equation}\label{eq:sameness-prod}
\bigl(\pm\tau_*\{\beta\},\; \gamma_c \pm \tau_*\{\alpha\} + \tau_*\{\beta\circ\eta\}\bigr) = \bigl(\pm\tau_*\{\beta'\},\; \gamma_{c'} \pm \tau_*\{\alpha'\} + \tau_*\{\beta'\circ\eta\}\bigr)
\end{equation}
(and similarly if the $\eta$-terms are absent).  We apply the map $F$ on the second co-ordinate of \eqref{eq:sameness-prod};  as $\ker F = Im([S^{4k},SG])$ we get $F(\gamma_c)=F(\gamma_{c'})$, that is $c\,t_{4k}=c'\,t_{4k}$, so $c=c'$.  The first co-ordinate gives $\tau_*\{\beta\}=\tau_*\{\beta'\}$.  Now,
$\tau_*\{\beta\}=\tau_*\{\beta'\}\ \Rightarrow\ \tau_*\{\beta\circ\eta\}=(\tau_*\{\beta\})\circ\eta=(\tau_*\{\beta'\})\circ\eta=\tau_*\{\beta'\circ\eta\}$.
So these terms cancel in the second co-ordinate of \eqref{eq:sameness-prod}, and with $\gamma_c=\gamma_{c'}$ this forces $\tau_*\{\alpha\}=\tau_*\{\alpha'\}$.  Note that the conclusion is the same whether or not the terms involving $\eta$ appears.\\ \\
($\Leftarrow$)\; Conversely, let $c=c'$, $\tau_*\{\alpha\}=\tau_*\{\alpha'\}$ and $\tau_*\{\beta\}=\tau_*\{\beta'\}$.  Then $\tau_*\{\beta\circ\eta\}=\tau_*\{\beta'\circ\eta\}$ as above, so \eqref{eq:sum-prod} gives $i^*\eta\bigl(A_{\alpha,\beta,c}\bigr)=i^*\eta\bigl(A_{\alpha',\beta',c'}\bigr)$, and hence $[A_{\alpha,\beta,c}]=[A_{\alpha',\beta',c'}]$ in $hS(X)/\Theta_{8k-1}$ by \eqref{eq:bridge-prod}.  This proves the same-ness condition. \qedhere
\end{proof}
\begin{proof}[proof of (3)]
    Given $M\simeq S^{4k-1}\times S^{4k}$, pick any homotopy equivalence $q$, so that $[M,q]\in hS(X)$; by exhaustion $[M,q]=[A_{\alpha,\beta,c}]$ in $hS(X)/\Theta_{8k-1}$, i.e. $M\cong X_{l(c)} \# \Sigma$ for some $\Sigma$ — almost diffeomorphic to  $X_{l(c)}$, for some $c\in\mathbb{Z}$
\end{proof}

\section{Smooth manifolds homotopy equivalent to $S^{4k}\times S^{4k}$}

We classify smooth closed oriented $8k$-manifolds homotopy equivalent to $S^{4k}\times S^{4k}$. Note that we have a surgery exact sequence $0 \to hS(S^{4k} \times S^{4k}) \xrightarrow{\eta} [S^{4k} \times S^{4k},G/O] \xrightarrow{\sigma} L_{8k}(e)$. First we would try to figure out the $Im(\eta)$. Then we would try to find inverses of all elements in $Im(\eta)$. The inverses of elements in the $Im([S^{4k} \times S^{4k},SG]) \cap Im(\eta)$ will be pinch maps. And recall from the preceding section, we had the family of rank-$4k$ oriented vector bundles $\xi_{l(c)} \to S^{4k}$ with $l(c) = c \cdot a_k(2k-1)!\operatorname{den}(B_k/4k)$. The elements of $Im(\eta) \notin Im([S^{4k} \times S^{4k},SG])$ have inverses coming from a coned off plumbing of two $D(\xi_{l(c)})$. The main results of this section are the following:

\begin{theorem}[Classification]\label{thm:classification}
Every smooth closed oriented $8k$-manifold homotopy equivalent to $S^{4k}\times S^{4k}$ is almost diffeomorphic to some $X_{(u,v)}$ for pairs $(u,v) \in [S^{4k}, G/O] \times [S^{4k}, G/O]$ satisfying  $\partial(u,v) = 0$.  
\end{theorem}

\begin{theorem}[The structure Set] \label{thm:s4ks4kstrset}
The structure set is given as follows:
\begin{enumerate} 
    \item The collection $\{ [A_{c(u),c(v),x,y}  :\ X_{u,v} \xrightarrow{f_{c(u),c(v)}} S^{4k} \times S^{4k} \xrightarrow{P_{(x,y)}} S^{4k} \times S^{4k}] \ | \ x,y \in \pi_{8k}(S^{4k}); \ c(u), c(v) \in \mathbb{Z} \ and \ d_k \mid c(u)c(v) \}$ exhausts $hS(S^{4k} \times S^{4k})/\Theta_{8k}$. \\

    \item Two elements, $[A_{c(u),c(v),x,y}]$ and $[A_{c(u)',c(v)',x',y'}]$ in that collection  are the same iff, $c(u)=c(u)', c(v)=c(v)', \tau^*\Sigma^\infty(x) = \tau^*\Sigma^\infty(x'), \tau^*\Sigma^\infty(y) = \tau^*\Sigma^\infty(y')$. This is a consequence of their normal invariants. 
    
\end{enumerate}
          
\end{theorem}

\subsection{Computation of $Im(\eta)$}\label{sec:setup}

The cofibration $S^{4k}\vee S^{4k} \xhookrightarrow{j} X \xrightarrow{q} S^{8k}$ and the section $ p_1 $ and $ p_2$ yield a splitting $[S^{4k} \times S^{4k}, G/O] \cong [S^{4k}, G/O] \oplus [S^{4k}, G/O]\oplus [S^{8k}, G/O]$.  Every $\eta' \in [S^{4k} \times S^{4k}, G/O]$ is therefore uniquely written as a \emph{triplet} $(u,v,w)$. Let, $F$ be the composite
\begin{equation}
  F\colon \pi_{4k}(G/O)
  \;\xrightarrow{\;\varphi\;}\;
  \pi_{4k}(G/\textnormal{Top})
  \;\xrightarrow[\cong]{\;\theta^{\textnormal{Top}}\;}\;
  L_{4k}(\mathbb{Z})
  \;=\;
  \mathbb{Z},
\end{equation}
where $\varphi$ is the forgetful map from the fibration $\textnormal{Top}/O \to G/O \to G/\textnormal{Top}$ and $\theta^{\textnormal{Top}}$ is the topological surgery obstruction isomorphism.  The long exact homotopy sequence gives $\operatorname{Im} F = t_{4k}\,\mathbb{Z}$ and $\ker F = Im([S^{4k},\textnormal{Top}/O]) = \operatorname{Im} ([S^{4k},SG])$, (see \cite[Thm.~5.4]{levine-lectures}).  Choosing $e_0 \in \pi_{4k}(G/O)$ with $F(e_0) = t_{4k}$, we obtain a splitting
\begin{equation}
  \pi_{4k}(G/O) \;\cong\;  Im([S^{4k},SG]) \;\oplus\; \mathbb{Z}\cdot e_0.
\end{equation}

  For $u \in \pi_{4k}(G/O)$, define the \emph{reduced coefficient} $c(u) \in \mathbb{Z}$ by $F(u) = c(u)\,t_{4k}$; thus $c(u) = 0$ iff $u \in \ker F$.  

By \cite[Lem.~4.1]{crowley-sss}, the smooth surgery obstruction of the normal map corresponding to a triplet $(u, v, w)$ is
\begin{equation}
  \theta^{\textnormal{Diff}}(u, v, w)
  \;=\;
  F(u)\cdot F(v) \;+\; F(w)
  \;\in\;
  L_{8k}(\mathbb{Z})
  \;=\;
  \mathbb{Z},
\end{equation}
where $F(u)\cdot F(v)$ denotes the $L$-theory ring product $\alpha\colon L_{4k}(\mathbb{Z}) \times L_{4k}(\mathbb{Z}) \to L_{8k}(\mathbb{Z})$.

\begin{remark}\label{rem:product}
Under the identification $L_{4k}(\mathbb{Z}) = \mathbb{Z}$, the $L$-theory ring product satisfies $\alpha(a, b) = 8ab$ (integer multiplication). Consequently, $F(u)\cdot F(v) = 8\,c(u)\,c(v)\,t_{4k}^2$.
\end{remark}

Define the \emph{realisability obstruction}
\begin{equation}\label{eq:partial}
  \partial\colon [S^{4k}, G/O] \oplus [S^{4k}, G/O] \;\to\; bP_{8k} \;\cong\; \mathbb{Z}/t_{8k},
  \qquad
  \partial(u,v) \;:=\; 8\,c(u)\,c(v)\,t_{4k}^2 \!\!\!\pmod{t_{8k}}.
\end{equation}
By \cite[Thm.~1.5]{crowley-sss}, a pair $(u,v) \in [S^{4k}, G/O] \oplus [S^{4k}, G/O]$ arises as $j^*\eta(f)$ for some smooth homotopy equivalence $f\colon M \xrightarrow{\simeq} S^{4k}\times S^{4k}$ if and only if $\partial(u,v) = 0$.

\begin{theorem}[Realisable pairs]\label{thm:cases}
A pair $(u,v) \in [S^{4k}, G/O] \oplus [S^{4k}, G/O]$ is realisable in exactly the following three cases:
\begin{itemize}
  \item \emph{Type 1:} ($c(u) = c(v) = 0$): both $u,v$ are in $Im([S^{4k},SG])$.
  \item \emph{Type 2:} (exactly one of $c(u), c(v)$ zero): that is exactly one of $u$ and $v$ is in $Im([S^{4k},SG])$. 
  \item \emph{Type 3} ($c(u), c(v) \neq 0$): iff $d_k \mid c(u)\,c(v)$, where $d_k := t_{8k}/\gcd(8\,t_{4k}^2,\, t_{8k})$.
\end{itemize}
\end{theorem}

For each realizable pair $(u,v)$, the $w$-component is determined up to an element in $Im([S^{4k},Top/O])$~\cite[Lem.~4.1]{crowley-sss}. So, any two homotopy equivalences realizing same pair $(u,v)$ differs only  by an action of homotopy sphere. So, the corresponding manifolds are almost diffeomorphic. And so, we only need to find an inverse realizing each $(u,v)$-pair. That is what we will do next and for the 3 types.

\subsection{Inverses of type 1 via pinch maps}\label{sec:sg-inverses}

\begin{proposition}\label{prop:sg-inverses}
For every pair $(u,v) \in Im([S^{4k}, SG] \oplus [S^{4k}, SG])$, there exists a pinch map homotopy equivalence $P_{(u,v)}\colon S^{4k}\times S^{4k} \xrightarrow{\simeq} S^{4k}\times S^{4k}$ with $i^*\eta(P_{(u,v)}) = (u,v)$.
\end{proposition}

\begin{proof}
Write $M := S_1^{4k}\times S_2^{4k}$ with factor inclusions $\iota\colon S_1^{4k} \hookrightarrow M$ (via $s \mapsto (s,*)$) and $\ell\colon S_2^{4k} \hookrightarrow M$ (via $s \mapsto (*,s)$), and projections $p_1, p_2$.  
\medskip
Choose $y \in \pi_{8k}(S^{4k})$ and form the composite
\begin{equation}
  \alpha \;:=\; \iota \circ y \;\colon\; S^{8k} \;\xrightarrow{\;y\;}\; S_1^{4k} \;\xrightarrow{\;\iota\;}\; M.
\end{equation}
Define the \emph{pinch map} $P(\alpha)\colon M \to M$ as
\begin{equation}
  P(\alpha)\colon M \;\xrightarrow{\;q'\;}\; M \vee S^{8k} \;\xrightarrow{\;\mathrm{Id}\,\vee\,\alpha\;}\; M \vee M \;\xrightarrow{\;\nabla\;}\; M,
\end{equation}
where $q'$ collapses a small embedded disk $D^{8k} \subset e^{8k}$ in the top cell. \emph{$P(\alpha)$ is a tangential homotopy equivalence}~\cite[Secs.~5,~7]{crowley-hambleton}.  Hence, $\eta(P(\alpha)) \in \operatorname{Im}\bigl(\tau_*\colon [M, SG] \to [M, G/O]\bigr)$. The proof of the proposition has two components; (A) Computing the restriction of the tangential normal invariant, (B) Establishing the surjectivity.

\medskip

\noindent\textbf{(A) Computing  $i^*\eta^t(P(\alpha))$ :}

\medskip

\noindent Since $M$, $S_1^{4k}$, and $S^{8k}$ are all stably parallelisable, their normal bundles (in a fixed ambient $\mathbb{R}^{8k+N}$, $N \gg 8k$) are trivial.  The maps $y$ and $\iota$ are therefore covered by bundle maps $b_y$ and $b_\iota$ between these trivial normal bundles.  We have the following ladder of bundle maps covering $S^{8k} \xrightarrow{y} S_1^{4k} \xrightarrow{\iota} M$:

\begin{equation}
    \begin{tikzcd}[column sep=3em, row sep=2.2em]
    S^{8k} \times \mathbb{R}^N
    \arrow[r, "y \times \mathrm{Id}"] \arrow[d, "\psi"', "\cong"]
    & S_1^{4k} \times \mathbb{R}^N
    \arrow[r, "\iota \times \mathrm{Id}"] \arrow[d, "\varphi"', "\cong"]
    & M \times \mathbb{R}^N
    \arrow[d, "\tilde\varphi"', "\cong"]
    \\
    \nu_{S^{8k}}
    \arrow[r, "b_y"] \arrow[d]
    & \iota^*(\nu_M)
     \arrow[r, "b_\iota"] \arrow[d]
    & \nu_M
    \arrow[d]
    \\
    S^{8k}
    \arrow[r, "y"]
    & S_1^{4k}
    \arrow[r, "\iota"]
    & M 
\end{tikzcd}
\end{equation}

\noindent Passing to Thom spaces of middle row and taking $(8k+2N)$-Spanier--Whitehead duals $D = D_{8k+2N}$ (using $\Sigma^NX_+ \simeq D(\mathrm{Th}(\nu_X))$ and $D(Th(\iota^*(\nu_M)))=Th(\nu(\iota))$ \cite[Sec.~7]{crowley-hambleton}) leads us to the following diagram:
\begin{equation}
\begin{tikzcd}[column sep=2.6em, row sep=2.4em]
S^{8k+N}
  \arrow[r] \arrow[d, dotted, "D"]
& \mathrm{Th}(\nu_{S^{8k}})
  \arrow[r, "\mathrm{Th}(b_y)"] \arrow[d, dotted, "D"]
& \mathrm{Th}(\iota^*\nu_M)
  \arrow[r, "\mathrm{Th}(b_\iota)"] \arrow[d, dotted, "D"]
& \mathrm{Th}(\nu_M)
  \arrow[d, dotted, "D"]
\\
S^N
& \Sigma^N S^{8k}_+
  \arrow[l, "\Sigma^N C_{s_0}"']
& \Sigma^N \mathrm{Th}(\nu_\iota)
  \arrow[l, "D(\mathrm{Th}(b_y))"']
& \Sigma^N M_+
  \arrow[l, "\Sigma^N \iota^!_+"']
\end{tikzcd}
\end{equation}
This is not a commutative diagram; the dotted arrows indicate Spanier--Whitehead duals.  Here $\iota^!_+\colon M_+ \to \mathrm{Th}(\nu_\iota)$ is the Pontryagin-Thom collapse for the embedding $\iota\colon S_1^{4k} \hookrightarrow M$, and ${+}$ goes to basepoint. \cite[Sec.~7]{crowley-hambleton}

By \cite[Lem.~7.4]{crowley-hambleton}, the tangential normal invariant of $P(\alpha)$ is $\eta^t(P(\alpha)) = [1] * g(\alpha) \in [M, SG]$, where $[1]*$ denotes the unit-component shift $[M, (\Omega^N S^N)_0]_* \xrightarrow{\cong} [M, SG]$, and $g(\alpha) \in [M, (\Omega^N S^N)_0]_*$ is
\begin{equation}
  g(\alpha)\colon M \;\xrightarrow{\;\iota^!\;}\; \mathrm{Th}(\nu_\iota) \;\xrightarrow{\;\mathrm{adj}\bigl(\Sigma^N C_{s_0} \circ D(\mathrm{Th}(b_y))\bigr)\;}\; \Omega^N S^N.
\end{equation}
We are interested in its adjoint $\tilde\eta^t(P(\alpha)) = \mathrm{adj}(g(\alpha)) \in \{M, S^0\} \cong [\Sigma^N M, S^N]$:
\begin{equation}
  \tilde\eta^t(P(\alpha))\colon
  \Sigma^N M
  \;\xrightarrow{\;\Sigma^N \iota^!\;}\;
  \Sigma^N \mathrm{Th}(\nu_\iota)
  \;\xrightarrow{\;D(\mathrm{Th}(b_y))\;}\;
  \Sigma^N S^{8k}_+
  \;\xrightarrow{\;\Sigma^N C_{s_0}\;}\;
  S^N.
\end{equation}

\medskip

\noindent We compute the restriction $i_2^*\,\mathrm{adj}(g(\alpha)) \in \{S_2^{4k}, S^0\} = \pi_{4k}^S$ by restricting the bottom composite to $\Sigma^N S_2^{4k} \xrightarrow{i_2} \Sigma^N M$. First, we need to understand $\Sigma^N\iota^!$ on homology because the domain and the range is wedge of spheres.
\medskip

 Write $S^{4k}_2$ as $S_2^{4k} = D_+^{4k}\cup_\partial D_-^{4k}$. The tubular neighbourhood  $\nu_\iota$ of $S_1^{4k}\times\{*\}$ in $M$ is $S_1^{4k} \times D^{4k}_+ \subset S_1^{4k} \times S_2^{4k}$  and its complement is
\begin{equation}
  A := M \setminus int{\bigl(S_1^{4k}\times D_+^{4k}\bigr)} = S_1^{4k}\times D_-^{4k}.
\end{equation}
The Pontryagin-Thom collapse is the quotient map
\begin{equation}
  \iota^! = q\colon M \longrightarrow M/A = \frac{S_1^{4k}\times S_2^{4k}}{S_1^{4k}\times D_-^{4k}} = \mathrm{Th}(\nu_\iota).
\end{equation}
The long exact sequence of the pair $(M, A)$ (with $H_*(M,A) \cong \tilde H_*(M/A)$) in degree $4k$ gives
\begin{equation}
  H_{4k}(A)
  \xrightarrow{\;i_*\;}
  H_{4k}(M) \cong \mathbb{Z} \oplus \mathbb{Z}
  \xrightarrow{\;(\iota^!)_* = q_*\;}
  H_{4k}(M/A) \cong \mathbb{Z}
  \xrightarrow{\;\partial\;}
  H_{4k-1}(A) = 0.
\end{equation}
From exactness we get, $(\iota^!)_* = 0$, on the $H_{4k}(S_1^{4k})$ summand and $(\iota^!)_*$ restricts to an isomorphism on the $H_{4k}(S_2^{4k})$ summand. Thus $(\iota^!)_* = P_2$ (projection onto the second component).

\medskip

Now that we know the map $\iota^!$, in Homology, we can try to find the map $adj(g(\alpha))$.  Taking Thom-space maps on the upper squares of the bundle ladder gives the diagram
\begin{equation}
\begin{tikzcd}[column sep=3em, row sep=2.6em]
\mathrm{Th}(\nu_{S^{8k}})
  \arrow[r, "\mathrm{Th}(\psi)", "\cong"'] \arrow[d, "\mathrm{Th}(b_y)"']
& S^{8k}_+ \wedge S^N
  \arrow[r, "q", "\simeq"'] \arrow[d, "\Sigma^N y_+"']
& S^{8k+N} \vee S^N
  \arrow[d, "\Sigma^N y \vee \mathrm{Id}"]
\\
\mathrm{Th}(\iota^*\nu_M)
  \arrow[r, "\mathrm{Th}(\varphi)", "\cong"']
& S^{4k}_+ \wedge S^N
  \arrow[r, "p", "\simeq"']
& S^{4k+N} \vee S^N
\end{tikzcd}
\end{equation}
Call the upper composite $m$ and the lower composite $h$; both are homotopy equivalences.  The $S_j^{4k}$-component of $i^*\tilde\eta^t(P(\alpha))$ is the composite
\begin{equation}
\Sigma^N S_j^{4k} \;\xrightarrow{\;i_j\;}\; \Sigma^N M \;\xrightarrow{\;\Sigma^N \iota^!\;}\; \Sigma^N \mathrm{Th}(\nu_\iota) \;\xrightarrow{\;D(\mathrm{Th}(b_y))\;}\; \Sigma^N S^{8k}_+ \;\xrightarrow{\;\Sigma^N C_{s_0}\;}\; S^N.
\end{equation}
which fits the following diagram.  And, by the following diagram we only need to compute the lower composite map.\cite[Ch.~6]{spanier-duality}
\begin{equation}
\begin{tikzcd}[column sep=3em, row sep=2.8em]
\Sigma^N M
  \arrow[r, "\Sigma^N \iota^!"]
& \Sigma^N \mathrm{Th}(\nu_\iota)
  \arrow[r, "D(\mathrm{Th}(b_y))"] \arrow[d, "D(h)"', "\simeq"]
& \Sigma^N S^{8k}_+
  \arrow[r, "\Sigma^N C_{s_0}"] \arrow[d, "D(m)"', "\simeq"]
& S^N
\\
\Sigma^N S_j^{4k}
  \arrow[u, "i_j"]
  \arrow[r, "p"']
& S^{4k+N} \vee S^{8k+N}
  \arrow[r, "\{y\} \vee \mathrm{Id}"']
& S^N \vee S^{8k+N}
  \arrow[ru, "\tilde c"', bend right=12]
\end{tikzcd}
\end{equation}
For $j=1$, $p=0$ and for $j=2$, $p=inclusion$. That gives, $i^*\tilde\eta^t(P(\alpha))=(0,\pm \{y\})$. Similarly, with $\beta: S^{8k} \xrightarrow{x} S_2^{4k} \hookrightarrow M$, we have, $i^*\tilde{\eta}^t(P(\beta))=(\pm \{x\},0)$.

\medskip

\noindent\textbf{(B) Surjectivity}
\medskip

\noindent The stabilisation $\pi_{8k}(S^{4k}) \to \pi_{4k}^S$ is surjective by the EHP exact sequence and finiteness of $\pi_{4k}^S$.  Hence as $y$ ranges over $\pi_{8k}(S^{4k})$, the stable class $\{y\}$ exhausts $\pi_{4k}^S \cong \pi_{4k}(SG) = [S_2^{4k}, SG]$.  We get,
\begin{equation}
  \bigl\{i_2^*\,\eta^t\bigl(P(\iota\circ y)\bigr)\bigr\}_{y \in \pi_{8k}(S^{4k})} \;=\; [S_2^{4k},\, SG].
\end{equation}

\noindent In the same way, taking the pinch maps on the other coordinate, we also get, 
\begin{equation}
  \bigl\{i_2^*\,\eta^t\bigl(P(l\circ x)\bigr)\bigr\}_{x \in \pi_{8k}(S^{4k})} \;=\; [S_1^{4k},\, SG].
\end{equation}

\medskip

\noindent Consider the composition $P(\iota\circ y) \circ P(\ell \circ x)\colon M \to M$.  By the composition formula \cite[Cor.~2.6]{mtw}:
\begin{equation}
  \eta^t(f\circ g) \;=\; \eta^t(f) \;+\; (f^{-1})^*\,\eta^t(g).
\end{equation}
Both pinch maps are the identity on the lower skeleton $M^{(4k)} = S^{4k}\vee S^{4k}$, so, $(P(\iota\circ y)^{-1})^*$ acts as the identity on $j^*([M, SG])$.  Restricting to the lower skeleton:
\begin{equation}
  j^*\eta^t\bigl(P(\iota\circ y)\circ P(\ell\circ x)\bigr)
  \;=\;
  j^*\eta^t\bigl(P(\iota\circ y)\bigr) \;+\; j^*\eta^t\bigl(P(\ell\circ x)\bigr).
\end{equation}
So, these compositions of pinch maps realize all of $[S_1^{4k}, SG] \oplus [S_2^{4k}, SG]$.  Passing to $[M, G/O]$ via $\tau_*$: $\eta\bigl(P(\iota\circ y)\circ P(\ell\circ x)\bigr)$ realises all of $\operatorname{Im}\bigl([S_1^{4k}, SG]\bigr) \oplus \operatorname{Im}\bigl([S_2^{4k}, SG]\bigr)$ in $[M, G/O]$. This proves (B)
\medskip

\noindent Collecting from (A) and (B), for every $(u,v) \in \operatorname{Im}\bigl([S_1^{4k}, SG]\bigr) \oplus \operatorname{Im}\bigl([S_2^{4k}, SG]\bigr)$, the pinch composition $P(\iota\circ y)\circ P(\ell\circ x)$ (with appropriate $y, x$) is a homotopy equivalence realising $(u,v)$ as lower-skeleton normal invariant. For any $(u,v)$, any choice of such map will be called $P_{(u,v)}$ and notice that the choice is not unique.  
\end{proof}

\subsection{Inverses of type 3 via plumbing}\label{sec:complement}

\begin{proposition}\label{prop:complement-inverses}
Let $(u,v) \in [S^{4k}, G/O]^{\oplus 2}$ with $c(u), c(v) \neq 0$ and $d_k \mid c(u)\,c(v)$ (Theorem~\ref{thm:cases}).  Then there exists a smooth homotopy equivalence $f_{c(u),c(v)}\colon N_{c(u),c(v),\phi} \xrightarrow{\simeq} S^{4k}\times S^{4k}$ with $F(i_1^*\eta(f_{c(u),c(v)})) = c(u)\,t_{4k}$ and $F(i_2^*\eta(f_{c(u),c(v)})) = c(v)\,t_{4k}$.
\end{proposition}

\begin{proof}
The construction follows Milnor's plumbing used in sphere products by \cite{crowley-sss}.

\medskip
\noindent\textbf{(A) construction of homotopy equivalences}
\medskip

First we construct the manifold from plumbing, then we will define the homotopy equivalences from each piece.

From the preceding section we have, for each $c \in \mathbb{Z}$, the rank-$4k$ oriented vector bundle $\xi_{l(c)} \to S^{4k}$ with $l(c) = c\cdot a_k(2k-1)!\operatorname{den}(B_k/4k)$, Euler class $e(\xi_{l(c)}) = 0$, and Pontryagin class $p_k(\xi_{l(c)}) = l(c)\cdot\omega$.  Given a pair $(c(u), c(v))$ with $c(u), c(v) \neq 0$, form the \emph{Milnor plumbing}
\begin{equation}
  W \;=\; W_{c(u),\,c(v)} \;:=\; D\bigl(\xi_{l(c(u))}\bigr) \;\natural\; D\bigl(\xi_{l(c(v))}\bigr)
\end{equation}
by choosing a basepoint $* \in S^{4k}$, trivialising both disk bundles over a small disk neighbourhood $D^{4k}_* \subset S^{4k}$ of $*$, and identifying $D^{4k}\times D^{4k} \subset D(\xi_{l(c(u))})$ with $D^{4k}\times D^{4k} \subset D(\xi_{l(c(v))})$ by exchanging fibre and base coordinates: $(x_1, x_2) \mapsto (x_2, x_1)$.  The result is a compact smooth $8k$-manifold $W$ with boundary $\partial W$, a homotopy sphere. Manifolds of the type $W$ --- compact, $(4k{-}1)$-connected, $8k$-dimensional, with boundary a homotopy sphere --- were classified by Wall \cite{wall-2n} using their intersection form $\lambda_W\colon H_{4k}(W)\times H_{4k}(W) \to \mathbb{Z}$ and stable tangential invariant $S\alpha(W)\colon H_{4k}(W) \to \pi_{4k-1}(SO) \cong \mathbb{Z}$.

For $W = W_{c(u),\,c(v)}$, these data are:
\begin{equation}
  H_{4k}(W) = \mathbb{Z}^2(x,y), \quad
  \lambda_W = \begin{pmatrix} 0 & 1 \\ 1 & 0 \end{pmatrix}, \quad
  S\alpha(W)(x) = s_k\,c(u), \quad
  S\alpha(W)(y) = s_k\,c(v),
\end{equation}
where $x = [S^{4k}_1]$ and $y = [S^{4k}_2]$ are the zero-section classes;  $\lambda_W$ is the standard hyperbolic form (the two zero sections meet transversely at the single plumbing point with intersection $+1$, and each has trivial self-intersection since $e(\xi_{l(c)}) = 0$);  and $s_k$ is the constant $l(1)/\lambda_k$, where $\lambda_k = \hat{p}_k(\beta_k)$ is the Pontryagin number of the Bott generator $\beta_k \in \pi_{4k-1}(SO)$.  In particular, $\sigma(W) = 0$.

\medskip

The boundary $\partial W$ is a homotopy $(8k{-}1)$-sphere, with class $[\partial W] \in \Theta_{8k-1}$.  By Wall's classification of almost-closed $(4k{-}1)$-connected $8k$-manifolds (the case $n = 4k$ even), this class depends only on the pair $(\tau, \chi^2)$ --- the signature $\tau = \sigma(W)$ and the self-intersection $\chi^2$ of the tangential invariant $\chi = S\alpha$~\cite[Cor.~2 and Thm.~3, pp.~171--172]{wall-2n}.  We therefore detect standardness of $[\partial W]$ in $\Theta_{8k-1}$ directly, by computing these terms and showing that they are equal with an almost closed manifold with standard sphere boundary.

For $W$ we have $\tau(W) = \sigma(W) = 0$ and
\begin{equation}\label{eq:mu-pk}
  p_k(W) \;=\; l(c(u))\,x^* + l(c(v))\,y^*.
\end{equation}
The self-intersection is computed from \eqref{eq:mu-pk} via the cup-product pairing on $[W,\partial W]$: writing $\hat x := \mathrm{PD}(x)$, the relations $\langle\hat x, x\rangle = x\cdot x = 0$ and $\langle\hat x, y\rangle = x\cdot y = 1$ force $\hat x = y^*$, and likewise $\hat y = x^*$; as the cup of Poincaré duals is dual to intersection, $(x^*)^2 = \hat y\smile\hat y \mapsto y\cdot y = 0$ and $x^*\smile y^* = \hat y\smile\hat x \mapsto y\cdot x = 1$~\cite[\S VI. Plumbing]{bredon}.  Hence by the above and ~\cite[eq. 13]{wall-2n}
\begin{equation}\label{eq:pk2-W}
  p_k^2[W,\,\partial W] \;=\; 2\,l(1)^2\,c(u)\,c(v),
  \qquad
  \chi^2(W) \;=\; \frac{2\,l(1)^2\,c(u)\,c(v)}{a_k^2 \cdot ((2k-1)!)^2} 
\end{equation}

We claim that $\partial W \cong S^{8k-1}$ if and only if $d_k \mid c(u)\,c(v)$.  If $d_k \mid c(u)c(v)$, then by the realizability theorem there is a closed smooth $N'$ and a homotopy equivalence $f'\colon N' \xrightarrow{\simeq} S^{4k}\times S^{4k}$ with $i^*\eta(f') = (u,v)$, hence $F(i_1^*\eta(f')) = c(u)\,t_{4k}$ and $F(i_2^*\eta(f')) = c(v)\,t_{4k}$.  Make $f'$ transverse to $S_1^{4k}\times\{*\}$ and set $N'_1 := f'^{-1}(S_1^{4k}\times\{*\})$, a closed $4k$-submanifold; then $g := f'|_{N'_1}\colon N'_1 \to S^{4k}$ is a degree-$1$ normal map. And, using the formula for surgery obstruction at dimension $4k$, Hirzebruch signature theorem, naturality of pontryagin classes; we get the following equation:
\begin{equation}
   c(u)\,t_{4k} = \; F(i_1^*\eta(f')) \; = \tfrac18\,\sigma(N'_1) \;= \frac{1}{8} s_k^L\,p_k(N'_1)[N'_1] =\; \frac{1}{8} s_k^L\,g^*p_k(i_1^*\bar{f}^*(\nu_{N'_1}))[N'_1] 
\end{equation}
  
with $s_k^L$ the coefficient of $p_k$ in the Hirzebruch class $L_k$. Now, we unpack $s^L_K, t_{4k}$ and do a $L.H.S = R.H.S$ computation to get $g^*p_k(i_1^*\bar{f}^*(\nu_{N'}))[N_1'] = l(c(u))$, which implies, $g^*p_k(i_1^*\bar{f}^*(\nu_{N'})) = l(c(u)) \cdot \omega_{N_1'}$, where $\omega_{N_1'}$ is the preferred generator. This \\ implies $i_1^*p_k(\bar{f}^*(\nu_{N'}))=l(c(u)) \cdot \omega$ and similarly, $i_2^*p_k(\bar{f}^*(\nu_{N'}))=l(c(v)) \cdot \omega$. Which means, $p_k(\bar{f}^*\nu_{N'})=l(c(u)) \cdot \omega_1 \; + \; l(c(v)) \cdot \omega_2 $ . Which leads to the following equation:
\begin{equation}
    p_k(N') \;=\; l(c(u))\,f'^*\omega_1 + l(c(v))\,f'^*\omega_2
\end{equation}

This gives $ 2\,l(1)^2\,c(u)\,c(v) = p_k^2(N') = p_k^2([N'^\circ, \partial N'^\circ]).$ Since $N' \simeq S^{4k}\times S^{4k}$, the intersection form on $N'$ is hyperbolic and $\sigma (N'^\circ) = \sigma(N') = 0$ (homotopy invariance of the form), and the cup-product computation (\cite[eq. 13]{wall-2n}; \eqref{eq:pk2-W}) applies the same to $N'$.  Thus for $N'^\circ := N' \setminus \mathring{D}^{8k}$,
\begin{equation}
     \tau(N'^\circ) = 0 = \tau(W),
  \qquad
  \chi^2(N'^\circ) =\frac{2\,l(1)^2\,c(u)\,c(v)}{a_k^2 \cdot ((2k-1)!)^2} = \chi^2(W)
\end{equation}
 
and, $\partial N'^\circ = S^{8k-1}$ being standard, Wall's theorem gives, \ $\partial W \cong S^{8k-1}$.  Conversely, if $\partial W \cong S^{8k-1}$ the construction below produces a homotopy equivalence $W \cup_\phi D^{8k} \xrightarrow{\simeq} S^{4k}\times S^{4k}$ realizing $(u,v)$, whence $d_k \mid c(u)c(v)$ by argument in \textbf{(B)} below.

Since $d_k \mid c(u)\,c(v)$, we have $\partial W_{c(u),c(v)} \cong S^{8k-1}$ by the above.  Choose a diffeomorphism $\phi\colon \partial W \xrightarrow{\cong} S^{8k-1} = \partial D^{8k}$ and define
\begin{equation}
  N_{c(u),\,c(v),\,\phi} \;:=\; W_{c(u),\,c(v)} \;\cup_\phi\; (-D^{8k}).
\end{equation}
Then $N$ is a closed smooth oriented $8k$-manifold.  
\medskip 

Now we construct the homotopy equivalence from $N_{c(u),\,c(v),\,\phi} \to S^{4k} \times S^{4k}$. $D(\xi_{l(c(u))})$ and $D(\xi_{l(c(v))})$ are both fiber homotopy equivalent to $S^{4k} \times D^{4k}$ because their sphere bundles are fiber homotopy trivial; call these fiber homotopy equivalences $f'_{c(u)},f'_{c(v)}$. Now we do plumbing on both sides and get $f'_{c(u),c(v)}\colon W_{c(u),c(v)} \to (S^{4k} \times S^{4k}) \setminus D^{8k}$. Both sides have boundary a standard sphere, so we cone it off by $D^{8k}$ and get
\begin{equation}
  f_{(c(u),c(v))}\colon N_{c(u),\,c(v),\,\phi} \xrightarrow{\simeq} S^{4k} \times S^{4k}.
\end{equation}
If we fix our choice of fiber homotopy equivalences of the disk bundles and the map $\phi$ to cone off the plumbed manifold, $f_{(c(u),c(v))}$ is fixed.

\medskip
\noindent \textbf{(B) The normal invariant.}
For simplicity we call the above map $f$ and domain manifold $N$, write $\bar{f}:S_1^{4k}\times S_2^{4k} \to N$ for a homotopy inverse of $f\colon N \to S_1^{4k}\times S_2^{4k}$. We restrict everything on $S_1^{4k}$ and let the normal map be $(g,b):(X,\tau_X) \to (S^{4k}_1,i_1^*\bar{f}^*(\tau_N))$. Then,
\begin{equation}\label{eq:Fi1-signature}
\begin{aligned}
     F(i_1^*\eta(f)) &= 1/8 \times (sgn(X))\\
     &= 1/8 \times L_k(p_1(X), \cdots, p_k(X))[X] \\
     &= 1/8 \times (Coeff(p_k)\  in \ L_k) \times p_k(X)[X] \\
     &=1/8 \times (Coeff(p_k)\  in \ L_k) \times g^*i_1^*\bar{f}^*p_k(\tau_N)[X]
\end{aligned}
\end{equation}
Now, $p_k(\tau_N) = p_k(\tau_W)= l(c(u))\,x^* + l(c(v))\,y^*$ (\S\ref{sec:complement}, equation~\eqref{eq:mu-pk}), and by construction of $\bar{f}$, $i_1^*\bar{f}^*p_k(\tau_N) = l(c(u))\,\omega$. And so, $g^*i_1^*\bar{f}^*p_k(\tau_N) = l(c(u)) \omega_X $ (as $g$ is degree 1), where $\omega$ and $\omega_X$ are preferred generator in cohomology.  Substituting in \eqref{eq:Fi1-signature}, $F(i_1^*\eta(f)) = c(u)\,t_{4k}$; by symmetry, $F(i_2^*\eta(f)) = c(v)\,t_{4k}$.
\end{proof}

\medskip

\begin{corollary}\label{cor:type3-realisation}
For every pair $(u,v) \in [S^{4k}, G/O]^{\oplus 2}$ with $c(u), c(v) \neq 0$ and $d_k \mid c(u)\,c(v)$, there is a smooth homotopy equivalence $f:N \xrightarrow{\simeq} S^{4k}\times S^{4k}$ realising the pair, i.e.\ with $i^*\eta(f) = (u,v)$.
\end{corollary}

\begin{proof}
Let $f\colon N \xrightarrow{\simeq} S^{4k}\times S^{4k}$ be the plumbing homotopy equivalence of Proposition~\ref{prop:complement-inverses}, so that $F(i_1^*\eta(f)) = c(u)\,t_{4k} = F(u)$ and $F(i_2^*\eta(f)) = c(v)\,t_{4k} = F(v)$.  Then $F$ kills each component of $(u,v) - i^*\eta(f)$, so
\begin{equation}
  (u,v) - i^*\eta(f) \in  \operatorname{Im}\bigl(\tau_*\colon [S^{4k}, SG]^{\oplus 2} \to [S^{4k}, G/O]^{\oplus 2}\bigr),
\end{equation}
say $(u,v) - i^*\eta(f) = (\tau_* x, \tau_* y)$ with $\tau\colon SG \to G/O$.  By Proposition~\ref{prop:sg-inverses} there is a pinch map $P_{(x,y)}$ with $i^*\eta(P_{(x,y)}) = (\tau_* x, \tau_* y)$.  Since $P_{(x,y)}$ is the identity on the lower skeleton, the composition formula \cite[Cor.~2.6]{mtw} makes $(P_{(x,y)}^{-1})^*$ the identity on $j^*[\,\cdot\,, G/O]$, hence
\begin{equation}\label{pinchofplumb}
  i^*\eta\bigl(P_{(x,y)} \circ f\bigr) \;=\; i^*\eta(P_{(x,y)}) + i^*\eta(f) \;=\; (\tau_* x, \tau_* y) + i^*\eta(f) \;=\; (u,v).
\end{equation}
Thus $P_{(x,y)} \circ f\colon N \xrightarrow{\simeq} S^{4k}\times S^{4k}$ realises $(u,v)$. And, this settles the realisation of type 3 case.
\end{proof}

\subsection{Inverses of type 2 via plumbing}\label{sec:remaining}

\begin{proposition}\label{prop:mixed-inverses}
Let $(u,v) \in [S^{4k}, G/O]^{\oplus 2}$ be a type-2 pair, i.e.\ exactly one of $c(u), c(v)$ is zero.  Then there is a smooth homotopy equivalence $f: N \xrightarrow{\simeq} S^{4k}\times S^{4k}$ realising $(u,v)$.
\end{proposition}

\begin{proof}
This is similar to the case of \S\ref{sec:complement} with one reduced coefficient zero.  Say $c(u) = 0$ and $c(v) \neq 0$ (the other case is symmetric).  By Theorem~\ref{thm:cases} the pair is realisable, since $\partial(u,v) = 0$ automatically.

We do the plumbing of \S\ref{sec:complement} with the pair $(0, c(v))$.  Here $l(0) = 0$, so $\xi_{l(0)} = \xi_0$ is the trivial bundle $S^{4k}\times \mathbb{R}^{4k}$, and
\begin{equation}
  W \;=\; D(\xi_0) \;\natural\; D\bigl(\xi_{l(c(v))}\bigr).
\end{equation}
Note that, both $\chi^2$ and $\tau$ are zero just the same as $S^{4k} \times S^{4k} \setminus int(D)$.
So $\partial W \cong S^{8k-1}$ always.
\end{proof}

\medskip
\begin{proof}[\textbf{Proof of Theorem~\ref{thm:classification}}]
    Let's call $N_{c(u),\,c(v),\,\phi} = X_{(u,v)}$ and notice that, $X_{(0,0)} = S^{4k} \times S^{4k}$. Proposition \ref{prop:sg-inverses},\ref{prop:complement-inverses},\ref{prop:mixed-inverses}, Corollary \ref{cor:type3-realisation} together exhibit, for every realisable pair $(u,v)$, a smooth homotopy equivalence $f_{(u,v)}$ with $i^*\eta(f_{(u,v)}) = (u,v)$.  By \S\ref{sec:setup} the $w$-component is then determined up to $\ker F_{8k}$, and two realisations of the same $(u,v)$ differ by connect-sum with some $\Sigma \in \Theta_{8k}$, hence the domains are almost diffeomorphic. So, for any $(u,v,w) \in Im(\eta)$, we have an inverse in the structure set whose domain is almost diffeomorphic to some $X_{(u,v)}$. This completes the proof of Theorem~\ref{thm:classification}.
\end{proof} 
\begin{proof}[\textbf{Proof of Theorem~\ref{thm:s4ks4kstrset}}] (The Structure Set)
\phantom\qedhere
\end{proof}
\begin{proof}[proof of (1)]
    Let, $[N,h] \in hS(S^{4k} \times S^{4k})$, such that, $i^*\eta([N,f])=(p,q) \in [S^{4k},G/O] \oplus [S^{4k},G/O]$. Then, by proposition \ref{prop:sg-inverses},\ref{prop:complement-inverses},\ref{prop:mixed-inverses}, Corollary \ref{cor:type3-realisation}  we know that, $\exists \ [A_{c(u),c(v),x,y}] \in hS(S^{4k} \times S^{4k})$ such that, $i^*\eta([A_{c(u),c(v),x,y}]) = (p,q)$. As, the $[S^{8k},G/O]$-component of realizable normal invariant is fixed by the last two co-ordinates up to an element in $Im(\tau_*: [S^{8k},SG] \to [S^{8k}, G/O])$; we have, $\eta([N,f]) - \eta([A_{c(u),c(v),x,y}] = c^*(x) = c^*(\eta([\Sigma_x]))$ for some $\Sigma_x \in \Theta_{8k}$ and, $x \in Im(\tau_*)$. Then, $\eta([N,f])= \eta(\Sigma_x \cdot[A_{c(u),c(v),x,y}])$. $\eta$ is one-one here for dimensional reason; So, $\Sigma_x \cdot [A_{c(u),c(v),x,y}]=[N,f]$. So, they have same equivalence class in $hS(S^{4k} \times S^{4k})/\Theta_{8k}$.  \phantom\qedhere     
\end{proof}
\begin{proof}[proof of (2)]
    Let, $\Sigma_x \cdot [A_{c(u),c(v),x,y}]=[A_{c(u)',c(v)',x',y'}]$. Then, $i^*\eta([A_{c(u),c(v),x,y}]) = \\ i^*\eta([A_{c(u)',c(v)',x',y'}])$. Using \eqref{pinchofplumb} we get, 
\begin{equation}\label{sum}
i^*\eta(f_{c(u),c(v)})+\bigl(\tau_*\Sigma^\infty(x),\tau_*\Sigma^\infty(y)\bigr)=i^*\eta(f_{c(u)',c(v)'})+\bigl(\tau_*\Sigma^\infty(x'),\tau_*\Sigma^\infty(y')\bigr)
\end{equation}
(Note that, \eqref{pinchofplumb} is true for all three cases and in the case with $c(u)=c(v)=0$; $f_{c(u),c(v)}=Id$). We apply the map $F$ on both side of \eqref{sum} to get, $c(u)=c(u)'$ and $c(v)=c(v)'$ by (\ref{prop:complement-inverses}). Now, in \eqref{pinchofplumb}, this forces $\tau_*\Sigma^\infty(x)=\tau_*\Sigma^\infty(x')$,
$\tau_*\Sigma^\infty(y)=\tau_*\Sigma^\infty(y')$. This proves the forward direction and it can be traced back step by step to prove the other direction too.

\end{proof}

\section{Smooth manifolds homotopy equivalent to $S^{4k}\times S^{4k+1}$}

We classify smooth closed manifolds homotopy equivalent to $X = S^{4k}\times S^{4k+1}$, $k \geq 2$, up to almost diffeomorphism.  In the smooth surgery exact sequence
\begin{equation}
L_{8k+2}(\mathbb{Z}) \xrightarrow{\;\partial\;} hS(X) \xrightarrow{\;\eta\;} [X, G/O] \xrightarrow{\;\sigma\;} L_{8k+1}(\mathbb{Z}) = 0,
\end{equation}
So, here $\eta$ is surjective and $\operatorname{Im}(\eta) = [X, G/O]$;  We have to realise each element of $[X, G/O]$.  The normal invariant of a homotopy equivalence $f\colon N \xrightarrow{\simeq} X$ decomposes into three pieces, each realised by a different construction:  the tangential part (the image of $[X, SG] \to [X, G/O]$ on the lower skeleton) by compositions of pinch maps;  the rank-one non-tangential $\mathbb{Z}$-summand of $[S^{4k}, G/O]$ by Novikov's sphere-bundle construction;  and the top-cell component $c^*[S^{8k+1}, G/O]$ by connected sum with a homotopy $(8k+1)$-sphere.  The main results of this section are the following:

\begin{theorem}[Classification]\label{thm:classification-odd}
Every smooth closed manifold $N$ homotopy equivalent to $S^{4k}\times S^{4k+1}$, $k \geq 2$, is almost diffeomorphic to $S(m\eta) \mathbin{\#} \Sigma'$ for some $m \in \mathbb{Z}$ and $\Sigma' \in \Theta_{8k+1}$, where $S(m\eta)$ is the total space of the sphere bundle of $m\eta \in \ker J \subset \pi_{4k}(BSO(4k+2))$, and $\eta$ is the generator of the kernel.  

\end{theorem}

\begin{theorem}[The Structure Set] \label{thm:s4ks4k+1strset}
The structure set is given as follows:
\begin{enumerate}
    \item The  collection: 
    $$ \{[P_{s\circ\beta} \circ P_{\ell\circ\alpha} \circ S(mh)\colon S(m\eta) \to X] \ | \ x \in \pi_{8k+1}(S^{4k}), y \in \pi_{8k+1}(S^{4k+1}) \ and \ m \in \mathbb{Z}\}$$ exhausts $hS(S^{4k} \times S^{4k+1})/\Theta_{8k+1}$.
    
    \item $[[P_{s\circ\beta} \circ P_{\ell\circ\alpha} \circ S(mh)]] = [[P_{s\circ\beta'} \circ P_{\ell\circ\alpha'} \circ S(m'h)]]$ iff $m=m' \ and \ \tau^*\Sigma^\infty(\alpha)=\tau^*\Sigma^\infty(\alpha'), \\ \tau^*\Sigma^\infty(\beta)=\tau^*\Sigma^\infty(\beta')$. 
\end{enumerate}
    
\end{theorem}

\subsection{splitting of normal invariant}\label{sec:setup-odd}

The CW structure $X \simeq (S^{4k}\vee S^{4k+1}) \cup_\rho e^{8k+1}$ gives a split short exact sequence
\begin{equation}\label{eq:split}
0 \to [S^{8k+1}, G/O] \xrightarrow{c^*} [X, G/O] \xrightarrow{i^*} [S^{4k}, G/O] \oplus [S^{4k+1}, G/O] \to 0
\end{equation}
  $F$ splits $[S^{4k}, G/O] = Im ([S^{4k}, SG)] \oplus \mathbb{Z}e_0$, for any $e_0 \in [S^{4k},G/O]$ with $F(e_0) = \pm t_{4k} = \pm |bP_{4k}|$. Notice that, $[S^{4k}, G/O] \oplus [S^{4k+1}, G/O] = Im([S^{4k}, SG]) \oplus \mathbb{Z}.e_0 \oplus Im([S^{4k+1}, SG])$. Any element in $[X,G/O]$ is of the form $[(\xi,h_\xi)]$, where $\xi$ is a stable vector bundle over $X$ with fiber homotopy trivial sphere bundle and $h_\xi$ is a trivialization of it's sphere bundle. Fix a choice of $h_\xi=S(h)$. Note that, we can choose $e_0$ to be $[(\eta, S(h)]$. With respect to this splitting every $x \in [X, G/O]$ therefore satisfies
\begin{equation}\label{eq:istar-decomp}
i^*(x) = (a,m,b) \in Im([S^{4k}, SG]) \oplus \mathbb{Z} \oplus Im([S^{4k+1}, G/O]),
\end{equation}
with $(a, b) \in Im([S^{4k}, SG] \oplus [S^{4k+1}, SG])$ and $m \in \mathbb{Z}$;  the remaining $c^*$-component lies in $[S^{8k+1}, G/O]$.

\subsection{Tangential inverses}\label{sec:tangential}

\begin{proposition}\label{prop:tangential}
Every element of $Im([S^{4k}, SG] \oplus [S^{4k+1}, SG])$ is realised as $ i^*\eta(P_g)$ for a composition of pinch maps $P_g = P_{s\circ\beta} \circ P_{\ell\circ\alpha}$ with $\alpha \in \pi_{8k+1}(S^{4k})$, $\beta \in \pi_{8k+1}(S^{4k+1})$.
\end{proposition}

\begin{proof}
By \cite[Thm.~3.3]{biswas-vs}:  $V_{4k+2, 2}$ has the same lower skeleton $S^{4k}\vee S^{4k+1}$ and a stably-trivial top-cell attaching map, so the computation of the pinch-map normal invariants there applies verbatim, realising classes $\alpha \in \pi_{8k+1}(S^{4k})$, $\beta \in \pi_{8k+1}(S^{4k+1})$ on the two factors. proceeding in the same way we prove the proposition.
\end{proof}

\subsection{Non-tangential inverses via the Novikov sphere bundle}\label{sec:novikov}

Let, $X=S^{4k} \times S^{4k+1}$ and $\eta\colon S^{4k} \to BSO(4k+2)$ generate $\ker J$ in $\pi_{4k}(BSO(4k+2)) \cong \mathbb{Z}$, of index $j_k = \mathrm{denom}(B_k/4k)$, with Pontryagin class $p_k(\eta) = a_k j_k (2k-1)!$.  A fibre-homotopy trivialisation $h\colon J\eta \simeq J\epsilon^{4k+2}$ induces a homotopy equivalence $S(h)\colon S(\eta) \xrightarrow{\simeq} X$ of sphere bundles \cite[Ex.~13.26]{ranicki-ags}.

\begin{proposition}\label{prop:higher-multiples}
For each $m \in \mathbb{Z}$, the $m$-fold sum $m\eta \in \pi_{4k}(BSO(4k+2))$ admits a fibre-homotopy trivialisation, inducing a homotopy equivalence $S(mh)\colon S(m\eta) \xrightarrow{\simeq} X$ with
\begin{equation}
i^*\eta(S(mh)) \;=\; (0,\,m,\, 0) \;\in\; Im([S^{4k}, SG]) \oplus \mathbb{Z} \oplus Im([S^{4k+1}, SG]),
\end{equation}

\end{proposition}

\begin{proof}
Let $S^{4k}\xrightarrow{\,P_m\,}S^{4k}$ be the degree $m$ map.
$m\eta\in\pi_{4k}(BSO(4k+2))$ is defined to be $P_m^{*}\eta$.
$S(mh)\colon S(m\eta)\xrightarrow{\ \simeq\ }S^{4k}\times S^{4k+1}$ is the fiber\ homotopy\ equivalence\ over $S^{4k}$ defined by $P_m^{*}S(h)$, pulling back at the base $S^{4k}$.
Now, $[(\eta,S(h))]\in\pi_{4k}\big((G/O)(4k+2)\big)$.
Then, by definition, $m\cdot[(\eta,S(h))]=P_m^{*}[(\eta,S(h))]=[(P_m^{*}\eta,\,P_m^{*}S(h))]=[(m\eta,\,S(mh))]$.
$\pi_1\colon S^{4k}\times S^{4k+1}\to S^{4k}$,\quad $i\colon S^{4k}\vee S^{4k+1}\hookrightarrow S^{4k}\times S^{4k+1}$.
Now, $i^{*}\eta\big(S(mh)\big)=i^{*}\pi_1^{*}\big([(m\eta,S(mh))]\big)=i^{*}\pi_1^{*}\big(m\cdot[(\eta,S(h))]\big)=m\cdot i^{*}\pi_1^{*}\big([(\eta,S(h))]\big)=m\cdot[(\eta,S(h))]=(0,m,0)$.
\end{proof}

\subsection{Proof of Main Theorem}\label{sec:assembly}

For $(a,m, b) \in Im([S^{4k}, SG]) \oplus \mathbb{Z} \oplus Im([S^{4k+1}, SG])$ put, $\Phi = P_g \circ S(mh)\colon S(m\eta) \to X$, with $P_g$ realising $(a,0, b)$ (Proposition~\ref{prop:tangential}) and $S(m \eta)$ realising $(0,m,0)$.  Since $P_g$ is the identity on the lower skeleton, Madsen's composition formula \cite[Cor.~2.6]{mtw} gives
\begin{equation}{\label{eq:mad-comp-s4kxS4kp1}}
i^*\eta(\Phi) \;=\; i^*\eta(P_g) + i^*\eta(S(mh)) \;=\; (a,0, b) + (0,m, 0) \;=\; (a,m, b)
\end{equation}
So, we can realize any element $(p, q) \in [S^{4k}, G/O] \oplus [S^{4k+1}, G/O]$ by these composite homotopy equivalences $\Phi$.

\begin{proposition}\label{prop:topcell}
For every $x \in [X, G/O]$ there exist a pinch map $P_g$, an integer $m \in \mathbb{Z}$, and a homotopy sphere $\Sigma' \in \Theta_{8k+1}$ such that
\begin{equation}\label{eq:fullPhi}
\Psi\colon S(m\eta) \mathbin{\#} \Sigma' \;\xrightarrow{h_{\Sigma'}}\; S(m\eta) \;\xrightarrow{S(mh)}\; X \;\xrightarrow{P_g}\; X
\end{equation}
satisfies $\eta(\Psi) = x$.
\end{proposition}

\begin{proof}
Choose $\Phi$ such that $i^*\eta(\Phi)=i^*(x)$, the residual $x - \eta(\Phi) = c^*(z) \in \operatorname{Im}(c^*)$ is realised by connected sum with a homotopy sphere $\Sigma' \in \Theta_{8k+1}$;  then $\eta(\Psi) = \eta(\Phi \circ h_{\Sigma'}) = \eta(\Phi) + c^*(z) = x$.
\end{proof}

\begin{proof}[\textbf{Proof of Theorem}~\ref{thm:classification-odd}]
Given a homotopy equivalence $q\colon N \to X$, set $x = \eta(q)$ and apply Proposition~\ref{prop:topcell} to obtain $\Psi$ with $\eta(\Psi) = x$.  By Novikov~\cite{novikov}, $N$ is almost diffeomorphic to the source of $\Psi$, namely $S(m\eta) \mathbin{\#} \Sigma'$, and hence also, $S(m\eta)$.
\end{proof}

\begin{proof}[\textbf{Proof of Theorem}~\ref{thm:s4ks4k+1strset}] (The Structure Set) \phantom\qedhere
\end{proof}
\begin{proof}[proof of (1)]
    Let, $X=S^{4k}\times S^{4k+1}$.
Take $[N,f]\in hS(X)$.
Then, $i^{*}\eta([N,f])=i^{*}\eta\big([S(mh)\circ P(l\circ\alpha)\circ P(s\circ\beta)]\big)$ for some $(\alpha,m,\beta)$. This means, 
$\eta([N,f])=\eta\big([S(mh)\circ P(l\circ\alpha)\circ P(s\circ\beta)]\big)+c^{*}(x)$, \quad where, $x\in\pi_{8k+1}(G/O)$,
and $x=\eta_S(\Sigma_x)$ for some homotopy sphere $\Sigma_x$, and,  $\eta_S\colon\Theta_{8k+1}\to\pi_{8k+1}(G/O)$ is the normal invariant map for sphere.
So, $\eta([N,f])=\eta\big(\Sigma_x\cdot[S(mh)\circ P(l\circ\alpha)\circ P(s\circ\beta)]\big)$, which implies, 
$[N,f]=(\Sigma_{bP}\cdot\Sigma_x)\cdot[S(mh)\circ P(l\circ\alpha)\circ P(s\circ\beta)]\in hS(X)$.
So, $[N,f]=[S(mh)\circ P(l\circ\alpha)\circ P(s\circ\beta)]\in hS(X)/\Theta_{8k+1}$.
This proves the realizability part. \phantom\qedhere
\end{proof}
\begin{proof}[proof of (2)]
    Let $\alpha\colon S^{8k+1}\to S^{4k}$, $\beta\colon S^{8k+1}\to S^{4k+1}$. Then, we have,
\begin{align}
&i^{*}\eta\big(P(l\circ\alpha)\circ P(s\circ\beta)\big)=\big(\tau_*\{\beta\},\ 0, \ \tau_*(\{\alpha\}+\{\beta\circ\eta\})\big) \ or \ \big(\tau_*\{\beta\},\ 0,  \ \tau_*\{\alpha\}\big) \\
\Rightarrow \  &i^{*}\eta\big(S(mh)\circ P(l\circ\alpha)\circ P(s\circ\beta)\big)=\big(\tau_*\{\beta\},\ m,\ \tau_*(\{\alpha\}+\{\beta\circ\eta\})\big) \ or \ \big(\tau_*\{\beta\},\ m,\ \tau_*\{\alpha\}\big)
\end{align}

\noindent Now, $\tau_*\{\beta\}=\tau_*\{\beta'\}\ \Rightarrow\ \tau_*\{\beta\circ\eta\}=(\tau_*\{\beta\})\circ\eta=(\tau_*\{\beta'\})\circ\eta=\tau_*\{\beta'\circ\eta\}$, by naturality of composition with $\eta$ applied to $\tau\colon SG\to G/O$.
So, in any case,
\begin{align*}
    i^{*}\eta\big(S(mh)\circ P(l\circ\alpha)\circ P(s\circ\beta)\big)=i^{*}\eta\big(S(m'h)\circ P(l\circ\alpha')\circ P(s\circ\beta')\big)
    \ \text{iff,} \ \ &m=m', \\ \tau_*\{\beta\}=\tau_*\{\beta'\},  \tau_*\{\alpha\}=\tau_*\{\alpha'\}&
\end{align*}
Now, observe that, 
\begin{align*}
    &[S(mh)\circ P(l\circ\alpha)\circ P(s\circ\beta)]=[S(m'h)\circ P(l\circ\alpha')\circ P(s\circ\beta')]\in hS(X)/\Theta_{8k+1}\\
    \Leftrightarrow\ &i^{*}\eta\big([S(mh)\circ P(l\circ\alpha)\circ P(s\circ\beta)]\big)=i^{*}\eta\big([S(m'h)\circ P(l\circ\alpha')\circ P(s\circ\beta')]\big)\\
    \Leftrightarrow\ &m=m',\ \tau_*\{\beta\}=\tau_*\{\beta'\}, \tau_*\{\alpha\}=\tau_*\{\alpha'\}
    \end{align*}
This proves the part (2). \qedhere

\end{proof}

\end{document}